\documentclass{svmult}

\usepackage{makeidx} 
\usepackage{graphicx}
\usepackage{multicol}       
\usepackage[bottom]{footmisc}
\usepackage{amssymb,amsmath}

\makeindex 

\begin{document}

\title*{A Lotka--Volterra type model analyzed through different techniques\thanks{Accepted
22-June-2023 for publication in a Springer Book.}}
\toctitle{A Lotka--Volterra type model analyzed through different techniques}
\titlerunning{A Lotka--Volterra type model analyzed through different techniques}

\author{Jorge Pinto$^1$ 
\and Sandra Vaz$^2$ 
\and Delfim F. M. Torres$^{3, 4,}$\thanks{Corresponding author: delfim@ua.pt}}
\tocauthor{Jorge Pinto, Sandra Vaz and Delfim F. M. Torres}
\authorrunning{J. Pinto, S. Vaz and D. F. M. Torres}

\institute{Department of Mathematics, University of Aveiro,\\
3810-193 Aveiro, Portugal, jorge.a.pinto@ua.pt
\and 
Center of Mathematics and Applications (CMA-UBI),\\
Department of Mathematics, University of Beira Interior,\\
6201-001 Covilh\~{a}, Portugal, svaz@ubi.pt
\and
Center for Research and Development in Mathematics and Applications (CIDMA), 
Department of Mathematics, University of Aveiro,\\  
3810-193 Aveiro, Portugal, delfim@ua.pt
\and
Research Center in Exact Sciences (CICE), 
Faculty of Sciences and Technology (FCT), 
University of Cape Verde (Uni-CV), 
7943-010 Praia, Cape Verde, delfim@unicv.cv}

\maketitle


\abstract{We consider a modified Lotka--Volterra model applied 
to the predator-prey system that can also be applied to other areas, 
for instance the bank system. We show that the model is  well-posed 
(non-negativity of solutions and conservation law) and study the local 
stability using different methods. Firstly we consider the continuous model, 
after which the numerical schemes of Euler and Mickens are investigated. 
Finally, the model is described using Caputo fractional derivatives. 
For the fractional model, besides well-posedness and local stability, 
we prove the existence and uniqueness of solution. Throughout the work 
we compare the results graphically and present our conclusions. 
To represent graphically the solutions of the fractional model  
we use the modified trapezoidal method that involves the 
modified Euler method.

\keywords{Lotka--Volterra model; 
Non-negativity of solutions; 
Stability; 
Mickens discretization; 
Fractional calculus.}
	
\medskip

\noindent {\bf MSC Classification 2020:} 34A08, 65L10, 65L12, 65L20, 65L70.}


\section{Introduction}
\label{sec1}

The Lotka--Volterra model arised in the middle of 20's 
of the XX century, by the mathematician Vito Volterra (1860--1940), 
who intended to explain the oscillatory levels of certain fish catches 
in the Adriatic sea \cite{Murray}. At the same time, the biophysicist  
Alfred Lokta (1880--1949) studied the same interaction predator-prey 
and published the book  ``Elements of Physical Biology'' \cite{Veron}.
The Lokta model was similar to the Volterra one, and 
the predator-prey model is nowadays called 
the Lotka--Volterra model \cite{MR4394645,MyID:505}.
	
The Lotka--Volterra model is a nonlinear population model 
that has several applications, from epidemiology \cite{RA,Delfim}, 
to biology  \cite{Murray} or economics \cite{KAMI}, among others. 
Sumarti, Nurfitriyana and Nurwenda, in 2014, applied the continuous  
Lotka--Volterra model with several interaction laws to the banking 
system and studied the local stability of the equilibrium points \cite{S}.  
Here we consider a particular case of the Lotka--Volterra model where 
the growth behavior of one of the species  is not exponential but logistic.
Precisely, the model that we are going to study is the following one:
\begin{align}
\label{int}
\left\{ \begin{array}{ll} 
{^c}D^\sigma D(t)& = \alpha D(t) \left(1 - \dfrac{D(t)}{C}\right) - p D(t) L(t),\\
{^c}D^\sigma L(t) & = p D(t) L(t) - \beta L(t), \quad 0 < \sigma \leq 1,
\end{array}\right.  
\end{align}
where $D(t)$ represents the prey population in the instant $t$ 
and $L(t)$ represents the predator population in time $t$. 
The value $C$ represent the carrying capacity or the ``ideal''  
capacity of the prey, $\alpha$ represents the growth rate of the prey, 
and $\beta$ represents the decay of the predators in the absence of the prey. 
On the other hand, the parameter $p$ represents the interaction rate between the species.
In fractional calculus, several notions of differentiation are available, 
such as the Gr\"{u}nwald--Letnikov derivative \cite{Tese}, 
the Riemann--Liouville derivative \cite{Tese}, 
or the Caputo derivative \cite{CAP,bra}, among others. 
Here we consider fractional differentiation in the Caputo sense:
the model is expressed in terms of the Caputo fractional derivative ${^c}D^\sigma$.  
If $\sigma =1$, then we recover the standard derivative of order one:
${^c}D^1 D(t)=\dot{D}(t)$ \cite{MR3443073}.

For most nonlinear models, it is not possible to obtain the exact solution, 
so numerical schemes are important to obtain numerical approximations 
and to obtain a graphical representation of the solutions. There are several 
numerical methods, some of them standard, while others are nonstandard. 
For instance, Euler and Heun methods are standard methods of discretization
while Mickens' method is an example of a nonstandard scheme 
\cite{AN,Mick94,Mick02,Mick05,Mick07}. 

In Section~\ref{sec2}, we analyze the equilibrium points of \eqref{int}
and their local stability in the case $\sigma = 1$.
Afterwards, in Section~\ref{sec3}, we discretize the model using the  
Euler numerical scheme \cite{AN}, and in Section~\ref{sec4} the Mickens numerical scheme 
\cite{Mick94,Mick02,Mick05,Mick07,Mick13}, analyzing the dynamical consistency 
with the continuous model \eqref{int}. 	
We proceed with Section~\ref{sec5}, using fractional calculus
to study \eqref{int} for $0 < \sigma < 1$: 
we prove the existence and uniqueness of a positive solution
(Proposition~\ref{EUF}) and the conservation law (Theorem~\ref{thm:consLaw}); 
we determine the local stability of the equilibrium points of the system
(Section~\ref{sub:sec:SA}); and we represent the solutions graphically 
(Section~\ref{sub:sec:GA}). For the fractional derivative, 
it is necessary to use suitable numerical schemes. Here we use 
the modified trapezoidal method that involves the modified Euler scheme 
\cite{frac18,frac08}. We end our work representing graphically some of the solutions 
for the different methods, comparing them, and with Section~\ref{sec6}
conclusion. All our numerical simulations were done using 
\textsf{Mathematica} version 12.1.


\section{The Modified Lotka--Volterra Model}
\label{sec2}

In this section, we determine the feasible region
for model \eqref{int} with $\sigma = 1$, 
the equilibrium points, and their local stability. 


\subsection{Model description}

In the model we have two populations. The variable $D(t)$ 
represents the prey population that, in the absence of predation, 
grows in the logistic way. The predator population, $L(t)$, 
in case of absence of any prey for sustenance, dies with exponential decay. 
The model is reasonable since in the physical world the populations are not 
infinite, so using the logistic function makes sense instead of exponential growth:
\begin{equation}
\label{banc}
\left\{\begin{array}{ll}
\dot{D} & = \alpha D \left(1 - \dfrac{D}{C}\right) - p D L,\\ \\
\dot{L} & = p D L - \beta L.
\end{array} \right.
\end{equation}

Throughout the text, we consider two hypothesis:
\begin{description}
\item [$(P_1)$] $0 < \alpha < \beta < p C <1$;
\item [$(P_2)$] all the initial conditions are positive.
\end{description}

The first condition is necessary for the existence of 
non-negative equilibrium points. We use the parameter $C$, the carrying capacity, 
in our computations, but it will be replaced by $C=1$ when necessary. 
In that case $D$ and $L$ represent a fraction of $C$. 

In our model, if $\alpha = 0$, then
\begin{align*}
\left\{\begin{array}{ll} 
\dot{D} & = - pD L,\\ \\
\dot{L} & = pD L - \beta L,
\end{array} \right.
\end{align*}
and the  prey population tends to extinction, which implies that the predator population 
will also be extinct because they do not have what to eat. 

Assuming the prey is the only food resource for the predators, 
and the predators do not migrate looking for other place to live, 
if $p = 0$, then
\begin{equation*}
\left\{\begin{array}{ll} 
\dot{D} &= \alpha D \left(1 - \dfrac{D}{C}\right),\\ 
\dot{L} &= - \beta L,
\end{array} \right.
\end{equation*}
and there is no interaction between the two species. This means that the 
predator population will be extinct and the prey will approach 
the carrying capacity $C$.


\subsection{Non-negativity of solutions and conservation law}

Model \eqref{banc} is a populational model so it has to satisfy some properties, 
namely the non-negativity of the solutions and the conservation law. 

Roughly speaking, Proposition~\ref{prop:01} asserts that the prey 
and predator populations are always non-negative.

\begin{proposition}
\label{prop:01}
Under hypotheses $(P_1)$ and $(P_2)$,
the solutions of system \eqref{banc}  
are non-negative for all $t > 0$, that is, 
the solutions are in	
\begin{equation*}
\Omega_+ = \left\{(D,L) \in (\mathbb{R}_{0}^{+})^2 
: D \geqslant 0, L \geqslant 0\right\}.
\end{equation*}
\end{proposition}

\begin{proof}
\smartqed	
Considering the first equation of  \eqref{banc},
\begin{equation*}
\dfrac{dD}{dt} = \alpha D \left(1 - \dfrac{D}{C}\right) - pDL 
\Leftrightarrow 
\dfrac{dD}{D} = \left[\alpha \left(1 - \dfrac{D}{C}\right) - pL \right] dt,
\end{equation*}
and integrating both sides we get
\begin{equation*}
D(T) = D_1 \exp^{\displaystyle \int_{0}^{T} \alpha 
\left(1 - \frac{D}{C}\right) - pL \, dt} 
\geqslant 0, \quad D_1 = e^{D_0}.
\end{equation*}
Analogously, considering the second equation of \eqref{banc},
\begin{equation*}
\dfrac{dL}{dt} = L(pD - \beta),
\end{equation*}
and we have
\begin{equation*}
L(T) = L_1 \exp^{\displaystyle \int_{0}^{T} pD - \beta \, dt} 
\geqslant 0, 
\quad L_1 = e^{L_0}.
\end{equation*}
Therefore, $(D,T) \in \Omega_{+}$.
\qed
\end{proof}

\begin{theorem}
Let $M = \max\{D(0),C\}$. If hypotheses $(P_1)$ and $(P_2)$ hold,
then the feasible region is given by
$$ 
\Omega = \left\{(D,L) \in (\mathbb{R}_{0}^{+})^2 : 
0 \leq D \leq M \text{ and } 0\leq D + L 
\leq \left(\dfrac{\alpha + 4 \beta}{4 \beta}\right)M \right\},
$$
that is, any solution that starts in $\Omega$ remains in  $\Omega$.
\end{theorem}

\begin{proof}
\smartqed
Considering the first equation of \eqref{banc},
\begin{equation*}
\dfrac{dD}{dt} = \alpha D \left(1 - \dfrac{D}{C}\right) - pDL 
\Rightarrow
\dfrac{dD}{dt}  \leq \alpha D \left(1 - \dfrac{D}{C}\right) 
\Leftrightarrow 
\dfrac{dD}{D \left(1- \frac{D}{C}\right)}  \leq \alpha dt.
\end{equation*}
Integrating both sides,
\begin{equation*}
\int \dfrac{1}{D} \, dD 
- \int \dfrac{-\frac{1}{C}}{1 - \frac{D}{C}} \, dD  
\leq \alpha t + k \Leftrightarrow
D(t) \leq \dfrac{C_1\exp^{\alpha t}}{1 + \frac{C_1}{C} \exp^{\alpha t}}.
\end{equation*}
If $D(0) = D_0$, then $C_1 = \dfrac{D_0}{1 - \frac{D_0}{C}}$. Thus,
\begin{equation*}
D(t) \leqslant \dfrac{D_0 \exp^{\alpha t}}{1 
- \frac{D_0}{C} (1 - \exp^{\alpha t})},
\end{equation*}
that is, $\underset{t \to \infty} {\lim}D(t) = C$ 
and $\underset{t \to 0}{\lim}D(t) = D_0$. Let $W(t) = D(t) + L(t)$. 
Then,
\begin{align*}
\dfrac{dW}{dt} 
&= \dfrac{dD}{dt} + \dfrac{dL}{dt}\\ 
&= \alpha D \left(1 - \dfrac{D}{C}\right) + \beta D - \beta D - \beta L\\ 
& = \alpha D \left(1 - \dfrac{D}{C}\right) + \beta D - \beta W.
\end{align*}
If $f(D) = \alpha D \left( 1 - \frac{D}{C}\right)$, 
then $\overline{D} = \frac{C}{2}$, which is a critical point that is the maximum. 
Therefore, $f(\overline{D}) = \frac{\alpha C}{4}$ and
\begin{align*}
\dfrac{dW}{dt} &\leq \dfrac{\alpha C}{4} + \beta D - \beta W, \\ 
\dfrac{dW}{dt} &\leq \left(\dfrac{\alpha + 4 \beta}{4} \right) M - \beta W, \\ 
\dfrac{dW}{dt} + \beta W &\leq \left(\dfrac{\alpha + 4 \beta}{4} \right) M.
\end{align*}
Solving the differential equation,
\begin{equation*}
W(t) \leq W_0 \exp^{-\beta t} 
+ \left(\dfrac{\alpha + 4 \beta}{4 \beta}\right) 
M (1 - \exp^{-\beta t})
\end{equation*}
and we have 
$\underset{t \to + \infty}{\lim} W(t) 
= \left(\dfrac{\alpha + 4 \beta}{4 \beta}\right) M$.
\qed
\end{proof}


\subsection{Stability analysis}

The equilibrium points of \eqref{banc} are
$$
\begin{cases}
\dot{D} & = 0,\\  
\dot{L} & = 0,\\
\end{cases} 
\Leftrightarrow 
\begin{cases}
\overline{D} & = 0,\\ 
\overline{L} & = 0,\\
\end{cases}
\vee 
\begin{cases}
\overline{D} & = C, \\ 
\overline{L} & = 0,\\
\end{cases} 
\vee 
\begin{cases}
\overline{D} & = \dfrac{\beta}{p}, \\ \\
\overline{L} & = \dfrac{\alpha}{p} \left(1 - \dfrac{\beta}{pC}\right).
\end{cases}
$$ 
We obtain three equilibrium points:
$$
e_1 =(0,0), 
\quad e_2 = (C,0), 
\quad e_3 = \left(\dfrac{\beta}{p}, 
\dfrac{\alpha}{p}\left(1 - \dfrac{\beta}{pC}\right)\right).
$$
The equilibrium $e_1$ means both populations tend to extinction;  
and $e_2$ means the prey population tend to $C$ and the predators population  
tend to extinction. It is important to note that the equilibrium point
$e_3$ only exists if  $1 - \frac{\beta}{pC} > 0$, which is one of the relations 
in $(P_{1})$, meaning that the prey population tends 
to $\frac{\alpha}{\beta}$ and the predator population tends 
to $\frac{\alpha}{p}\left(1 - \frac{\beta}{pC}\right)$.

To study the local stability, we linearize the system. 
The Jacobian matrix is 
$$ 
J = \left[ 
{\begin{array}{cc}
\alpha \left(1 - 2\dfrac{D}{C}\right) - pL  & -pD \\
pL & pD - \beta 
\end{array}} 
\right].
$$
In $e_1 =(0,0)$,
$$ 
J(e_1) = \left[ 
{\begin{array}{cc}
\alpha  & 0 \\
0 & - \beta 
\end{array}} 
\right]
$$
and so
\begin{equation*}
P(\lambda) = \det (J(e_1) - \lambda Id) 
= (\alpha - \lambda)(-\beta - \lambda). 
\end{equation*}
The eigenvalues are $\lambda_1 = \alpha$ 
and $\lambda_2 = - \beta$. By hypothesis $(P_1)$,  $\lambda_1 > 0$ 
and $\lambda_2 < 0$. This means $e_1$ is a \emph{saddle point}. 

In $e_2 = (C,0)$,
$$ 
J(e_2) 
= 
\left[ 
{\begin{array}{cc}
- \alpha  & -pC \\
0 & pC - \beta \\
\end{array}} 
\right]
$$
and
\begin{equation*}
P(\lambda) = \det (J(e_2) - \lambda Id) 
= (-\alpha - \lambda)(pC -\beta - \lambda).
\end{equation*}
The  eigenvalues are $ \lambda_1 = - \alpha$ 
and $\lambda_2 = - \beta + pC$. From $(P_1)$ it can be seen 
that $\lambda_1 < 0$ and $\lambda_2 > 0$. Once again, 
$e_{2}$ is a  \emph{saddle point}. 

For $e_3 = \left(\dfrac{\beta}{p}, 
\dfrac{\alpha}{p}\left(1 - \dfrac{\beta}{pC}\right)\right)$,
$$ 
J(e_3) 
= \left[ 
{\begin{array}{cc}
-\dfrac{\alpha \beta}{pC}  & -\beta \\
\alpha\left(1 - \dfrac{\beta}{pC}\right) & 0 
\end{array}} 
\right]
$$
and
\begin{equation*}
P(\lambda) = \det (J(e_3) - \lambda Id) 
= \lambda^2 + \dfrac{\alpha \beta}{pC} \lambda 
+ \beta \alpha \left(1 - \dfrac{\beta}{pC}\right).
\end{equation*}
By the  Routh--Hurwitz criterion and hypothesis $(P_1)$, all coefficients 
are greater than zero and we conclude that all the roots are negative. 
Therefore, $e_3$ is a \emph{sink} or \emph{asymptotically stable}. 


\subsection{Graphical analysis}

Throughout the work, we represent graphically the solutions, 
under different initial conditions, and compare the different methods, 
visualizing the similarities and differences. In our simulations,
for comparison reasons, the parameters will always be 
$C=1$, $a=0.05$, $b=0.3$, and $p=0.4$.

In Figure~\ref{fig2}, the continuous model \eqref{int} 
is represented using different initial conditions.
\begin{figure}[ht!]
\centering
\includegraphics{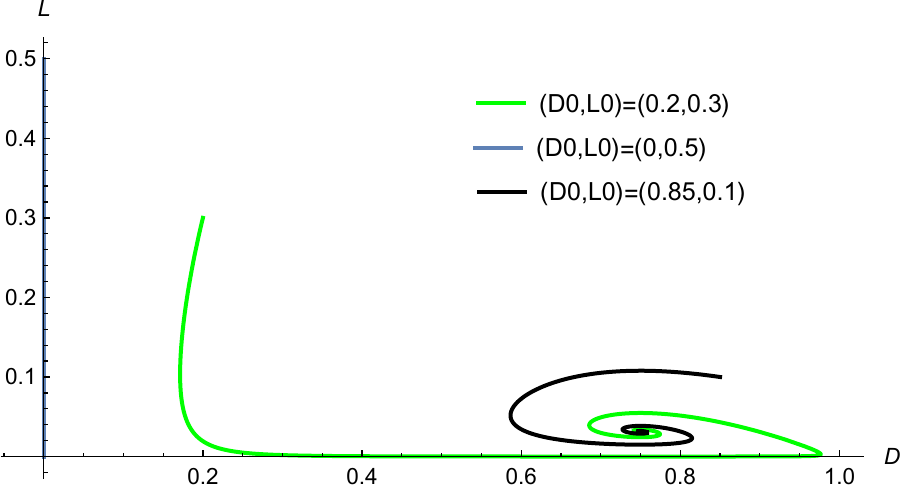}
\caption{Continuous model \eqref{int}
with different initial conditions.}\label{fig2}
\end{figure}
We see that if $(D_0, L_{0}) = (0.2, 0.3)$, then the predators 
tend to extinction and the prey population increase. 
When the prey tend to the carrying capacity, then the predators 
population increase, converging to $e_{3}$. 
On the other hand, if $(D_0,L_{0}) = (0, 0.5)$, and since the prey are extinct, 
the predators tend to extinction because of absence of food. Finally,
if $(D_0,L_{0}) = (0.85,0.1)$, then the population of prey and predators 
tend to  $e_3$.  

Figure~\ref{fig9} corresponds to initial conditions  
$(D_0, L_{0}) =(0.2, 0.3)$ and will be used 
for comparison with the other models in the sequel. 
\begin{figure}[ht!]
\centering
\includegraphics[width=9 cm]{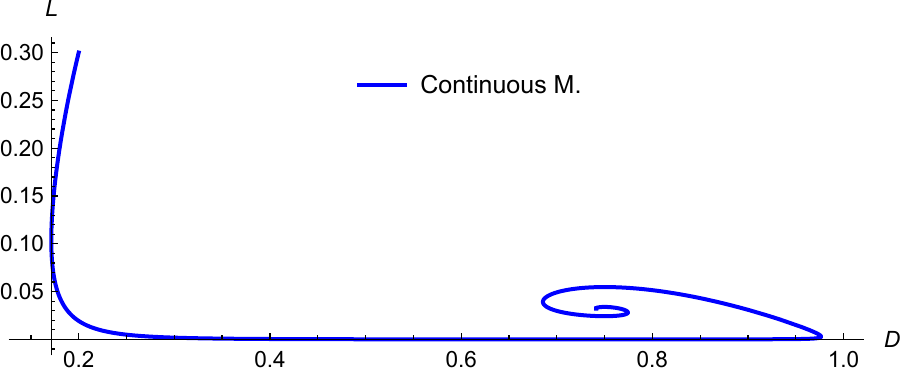}
\caption{Continuous model \eqref{int}
with $(D_0,L_{0}) = (0.2,0.3)$.}\label{fig9}
\end{figure} 


\section{Euler's Numerical Scheme}
\label{sec3}

Leonhard Euler (1707--1783) presented his method, 
now known as the Euler method, in \textit{Institutionum calculi integralis} (1768-1770). 
Euler's method is an iterative process for approximating the solution of a problem. 
In other words, given an ordinary first-order differential equation 
or system, with known initial condition, Euler's method give us a way to determine 
an approximated numerical solution. Let us consider $y'=f(t,y)$, $y(t_0)=y_0$, 
and $t \in [a,b]$, and let us compute 
the Taylor series for $t = t_1 = t_0 + h$:
\begin{equation*}
y(t_1) \approx y_1 = y_0 + hf(t_0,y_0) + \dfrac{h^2}{2} f'(t_0,y_0) 
+ \ldots + \dfrac{h^k}{k!} f^{(k-1)} (t_0,y_0), 
\end{equation*}
where $f^{(j)}(t_0,y_0) = \dfrac{d^j f}{dt^j}(t_0,y_0)$.
The following explicit method with uniform step-size $h$
allow us to obtain approximated solutions $y_i \approx y(t_i)$:
\begin{equation} 
\label{euler}
\begin{split}
y_0 &= y(t_0),\\
y(t_{i+1}) &\approx y_{i+1} = y_i + h f(t_i,y_i), 
\quad i=0, \ldots, n.
\end{split}
\end{equation}
The scheme \eqref{euler} is precisely the Euler method. 
The step-size $h = t_{i+1} - t_i$, $i = 0, \ldots, n$, is uniform 
and $t_i = t_0 + ih$, $i = 0, \ldots, n$, where $h = \frac{b-a}{n}$. 
For physical models, this method is not frequently used because 
it may be difficult to prove the non-negativity of the solutions. 
Moreover, it is also known to cause numerical instabilities 
and present dynamical inconsistency with the continuous model
\cite{Elaydi01,Mick02,Shi}.


\subsection{Model discretization}

Using the concepts just introduced,
$$
\begin{cases}
\dot{D}(t) \approx \dfrac{D(t+h) - D(t)}{h},\\ 
\dot{L}(t) \approx \dfrac{L(t+h)-L(t)}{h},
\end{cases} 
$$
and we get
\begin{align}
\label{EE}
& \begin{cases} 
\dfrac{D_{i+1} - D_i}{h} 
= \alpha D_i \left(1 - \dfrac{D_i}{C}\right) - pD_iL_i,\\ 
\dfrac{L_{i+1} - L_i}{h} 
= pD_iL_i - \beta L_i,
\end{cases} \nonumber \\ 
\Leftrightarrow
&\begin{cases} 
D_{i+1} = D_i \left(\alpha h 
\left(1 - \dfrac{D_i}{C}\right) - phL_i + 1 \right), \\ 
L_{i+1} = L_i \left(phD_i - \beta h + 1\right),  
\end{cases}
\end{align}
$i \geq 0$. System \eqref{EE} is the explicit discrete model 
by the standard Euler's scheme.


\subsection{Non-negativity of solutions and conservation law}

Now we prove that the discrete system \eqref{EE} satisfies 
the non-negativity condition and the conservation law. 
First we prove an auxiliary lemma under the following hypothesis:
\begin{description}
\item[$(H_1)$] $1 - \beta h > 0$.
\end{description}

\begin{lemma}
Let $h$ be the step-size and suppose 
that $(P_{1})$ and $(H_{1})$ hold. 
Then, 
$$
\dfrac{1+\alpha h}{p h} \geq C.
$$
\end{lemma}

\begin{proof}
\smartqed	
From hypotheses $(P_{1})$ and $(H_{1})$, 
$0>1 + \alpha h - \beta h > 1+ \alpha h - C p h$
because $\beta h <  C p h$.
\qed
\end{proof}

\begin{theorem}
If hypotheses $(P_1)$, $(P_2)$ and $(H_1)$ hold,
then all solutions of system \eqref{EE} are non-negative  
for all $n \geq 0$ with the feasible region being given by
$$
\Omega = \left\{(D,L) \in (\mathbb{R}_{0}^{+})^2 
: W_n = D_n + L_n \leq C\right\}.
$$
\end{theorem}

\begin{proof}
\smartqed	
By hypotheses $(P_{1})$ and $(H_1)$, we easily see 
that the second equation of \eqref{EE},
\begin{equation*}
L_{n+1} = L_n\left(1 -  \beta h + phD_n\right), 
\end{equation*}
satisfy $L_{n+1} \geqslant 0$ for all $n$. 
Regarding the first equation, by $(P_1)$ it follows that
\begin{align*}
D_{n+1} 
&= D_n \left(1 +\alpha h - h 
\left(\dfrac{\alpha}{C}D_n+pL_n\right)\right)\\
& > D_n(1 + \alpha h - ph(D_n+L_n))\\
&> D_n(1 + \alpha h - phW_n).
\end{align*}
Therefore, $D_{n+1} \geq 0$ if, and only if, 
$1 + \alpha h - ph W_n  \geq 0$, that is, 
$W_n \leq \frac{1 + \alpha h}{ph}$.
Let $W_n  = D_n + L_n$. By \eqref{EE}, we have 
\begin{align*}
& \dfrac{W_{n+1} - W_n}{h} 
= \alpha D_n \left(1 - \dfrac{D_n}{C}\right) - \beta L_n\\
\Leftrightarrow	& \dfrac{W_{n+1} - W_n}{h} 
= \alpha D_n \left(1 - \dfrac{D_n}{C}\right) 
- \beta L_n - \beta D_n + \beta D_n\\
\Leftrightarrow & \dfrac{W_{n+1} - W_n}{h} 
= \alpha D_n \left(1 - \dfrac{D_n}{C}\right) 
- \beta W_n + \beta D_n
\end{align*}
and from $(P_1)$ it follows that
\begin{align*}
W_{n+1} 
&\leq (1 - \beta h)W_n + \beta h D_n 
\left(1 - \dfrac{D_n}{C}\right) + \beta h D_n\\
&= (1 - \beta h)W_n + \beta h D_n \left(2 - \dfrac{D_n}{C}\right).
\end{align*}
Let $f(D) = \beta h D \left(2 - \frac{D}{C}\right)$. 
Then the critical point is $\overline{D} = C$, which is a maximum, 
and $f(\overline{D}) = \beta h C$. Moreover, 
$W_{n+1} \leq (1 - \beta h) W_n + \beta h C$, that is,
\begin{equation*}
W_{n+1} \leq (1 - \beta h)^{n+1} W_0 
+ C \left(1 - (1- \beta h)^n\right).
\end{equation*}
Thus, $\underset{n \to + \infty}{\lim} W_{n+1} = C$, 
which means that $W_n \leq C \leq \dfrac{1+\alpha h}{p h}$ 
for all $n \geqslant 0$. We conclude that the feasible region is
\begin{equation*}
\Omega = \left\{(D,L) \in {\mathbb{R}_0^{+}}^2 
: 0 \leq W_n = D_n + L_n \leq C\right\}
\end{equation*} 
and that \eqref{EE} satisfies the non-negativity condition 
and the conservation law if $(H_1)$ is verified. 
\qed
\end{proof}


\subsection{Stability analysis}	

The stationary points of the Euler discrete system \eqref{EE} 
are obtained from
$$
\begin{cases}
F(D^*) = D^*,\\ 
F(L^*) = L^*,\\
\end{cases} 
\Leftrightarrow 
\begin{cases}
D^* = 0,\\ 
L^* = 0,\\
\end{cases}
\vee 
\begin{cases}
D^* = C, \\ 
L^* = 0,\\
\end{cases}
\vee 
\begin{cases}
D^* = \dfrac{\beta}{p},\\
L^* = \dfrac{\alpha}{p} 
\left(1 - \dfrac{\beta}{pC}\right),
\end{cases} 
$$
that is,
$$
e_1 =(0,0), 
\quad e_2 = (C,0), 
\quad e_3 = \left(\dfrac{\beta}{p}, 
\dfrac{\alpha}{p}\left(1 - \dfrac{\beta}{pC}\right)\right),
$$
which are equal to the ones of the continuous model \eqref{banc}.

To study the local stability, we linearize the system. 
The Jacobian matrix is given by
$$ 
J = \left[ 
{\begin{array}{cc}
\alpha h \left(1 - \dfrac{2D}{C}\right) - phL + 1  & -phD \\
phL & phD - \beta h + 1 \\
\end{array} } \right].
$$
For $e_1=(0,0)$, the Jacobian matrix is
$$ 
J(e_1) 
= 
\left[ 
{\begin{array}{cc}
1 + \alpha h  & 0 \\
0 & 1 - \beta h  \\
\end{array} } 
\right]
$$
and
\begin{equation*}
P(\lambda) = \det (J(e_1) - \lambda Id) 
= (1 + \alpha h - \lambda)(1 -\beta h - \lambda). 
\end{equation*}
The eigenvalues are $ \lambda_1 = 1 + \alpha h$ and $\lambda_2 = 1 - \beta h$. 
By $(P_1)$, $\lambda_1 > 1$, and by $(H_1)$, $ \lambda_2 < 1$. 
We conclude that $e_1$ is a \emph{saddle point}. 

For $e_2 = (C,0)$, the Jacobian matrix is
$$ 
J(e_2) 
= 
\left[ 
{\begin{array}{cc}
1 - \alpha h  & -phC \\
0 & phC - \beta h + 1 \\
\end{array} } \right]
$$
and
\begin{equation*}
P(\lambda) = \det (J(e_2) - \lambda Id) 
= (1 -\alpha h- \lambda)(phC -\beta h + 1 - \lambda). 
\end{equation*}
The eigenvalues are  $\lambda_1 = 1 - \alpha h$ 
and $\lambda_2 = - \beta h + phC + 1$. By $(P_1)$, 
we can see that $\lambda_2 > 1$ and $\lambda_1 < 1$,
and thus $e_2$ is a \emph{saddle point}.

For $e_3 = \left(\dfrac{\beta}{p}, 
\dfrac{\alpha}{p}\left(1 - \dfrac{\beta}{pC}\right)\right)$, 
the Jacobian matrix is
$$ 
J(e_3) 
= 
\left[ 
{\begin{array}{cc}
1 -\dfrac{\alpha h \beta}{pC}  & -\beta h \\ 
\alpha h \left(1 - \dfrac{\beta}{pC}\right) & 1 \\
\end{array} } 
\right]
$$
and
\begin{equation*}
P(\lambda) = \det (J(e_3) - \lambda Id) 
=  \left(1 - \dfrac{\alpha \beta h}{pC}  - \lambda \right)
(1 -  \lambda) + \alpha \beta h^2  \left(1 - \dfrac{\beta}{pC}\right). 
\end{equation*}
From Schur--Cohn criterion for quadratic polynomials $P(x)$,  
if $P(1) > 0$,  $P(-1) > 0$, and  $\mid P(0) \mid < 1$, 
then all the roots  are inside  the unit circle 
(see \cite{Elaydi, Elaydi01}).
Here we have $P(1) = \alpha \beta h^2 \left(1 - \frac{\beta}{pC}\right)$, 
so by $(P_1)$ we can conclude that $P(1)>0$; while 
$$
P(-1) = 2 \left(2 - \dfrac{\alpha \beta h}{pC}\right) 
+ \alpha \beta h^2 \left(1 - \dfrac{\beta}{pC}\right)
$$
and by $(P_1)$, since  
$\dfrac{\beta}{pC} < 1 
\Leftrightarrow 2 - \dfrac{\alpha \beta h}{pC} > 2 - \alpha h$, 
one has $P(-1) >0$. Moreover, 
$$
P(0) = 1 - \dfrac{\alpha \beta h}{pC} 
+ \alpha \beta h^2 \left(1 - \dfrac{\beta}{pC}\right).
$$
We need to show that $P(0) > -1 $. By $(P_1)$, and since 
$$
1 - \dfrac{\alpha \beta h}{pC} > 1 - \alpha h> 1-\beta h>0,
$$  
we have 
$$
P(0)=1 - \dfrac{\alpha \beta h}{pC} 
+ \alpha \beta h^2 \left(1 - \dfrac{\beta}{pC}\right)
>1-\beta h +\alpha \beta h^2 \left(1 - \dfrac{\beta}{pC}\right)>0 >-1
$$
and to have $P(0)<1$ we need to show that 
\begin{equation} 
\label{anulamento}
- \frac{\alpha \beta}{pC}h + \alpha \beta 
\left(1 - \frac{\beta}{pC}\right) h^2 < 0.
\end{equation}
Because \eqref{anulamento} is a polynomial $g(h)$
of degree $2$ with roots
\begin{equation*}
h_1 = 0
\quad  \vee \quad 
h_2 = \dfrac{1}{pC - \beta}, 
\end{equation*}
we have $P(0) <1$ as long as $0 < h < h_2$.
Therefore, we have shown that $ \mid P(0) \mid < 1$ 
as long as $0< h < h_2$ and, in this case, $e_3$ 
is a \emph{sink} or \emph{asymptotically stable}.

Under Euler's numerical scheme, the conservation law 
is not straightforward and we need to impose some 
conditions to have dynamical consistency with the continuous model. 
In order for $e_{3}$ to be asymptotically stable, the step-size 
must be smaller than $h_{2}= \dfrac{1}{pC - \beta} $. By $(P_{1})$, 
all parameters are less than 1, so it is not a difficult condition 
to be attained.


\subsection{Graphical analysis}

For the numerical scheme of Euler, we take
the time interval to be $[0, 300]$ and the 
step-size as $h = 0.25$.

In Figure~\ref{fig3}, some solutions of \eqref{EE}, 
for different initial conditions, are plotted.
\begin{center}
\begin{figure}[ht!]
\begin{center}
\includegraphics{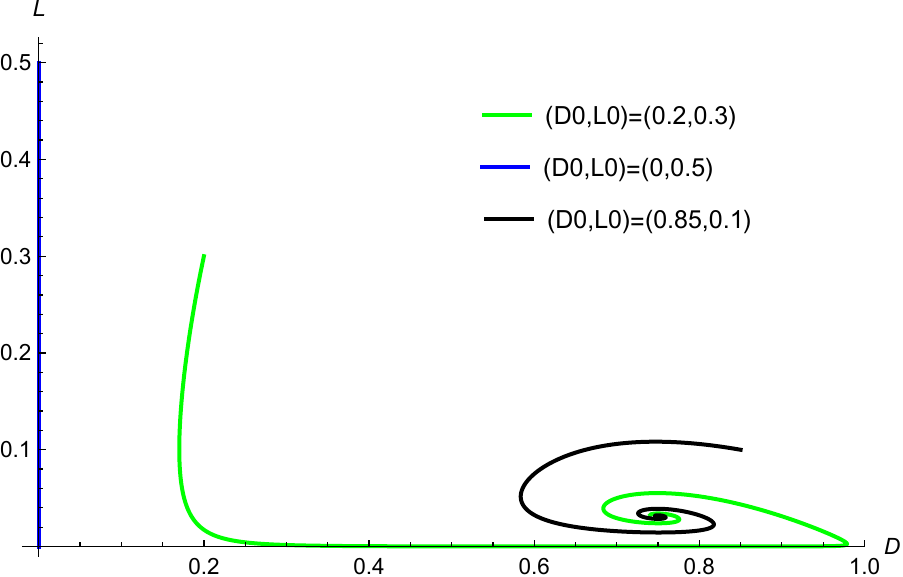}
\caption{Solutions to Euler's discrete model \eqref{EE} 
with different initial conditions.\label{fig3}}
\end{center}
\end{figure}   
\end{center}
Figure~\ref{fig3} is similar to Figure~\ref{fig2}.

In Figure~\ref{fig8}, we compare the solution 
of the continuous model \eqref{banc}
with the one obtained by the Euler method 
with initial conditions $(D_0, L_{0})=(0.2, 0.3)$. 
There are only a few mild differences between the plots,
which are qualitatively the same.
\begin{center}
\begin{figure}[ht!]
\begin{center}
\includegraphics[width=9 cm]{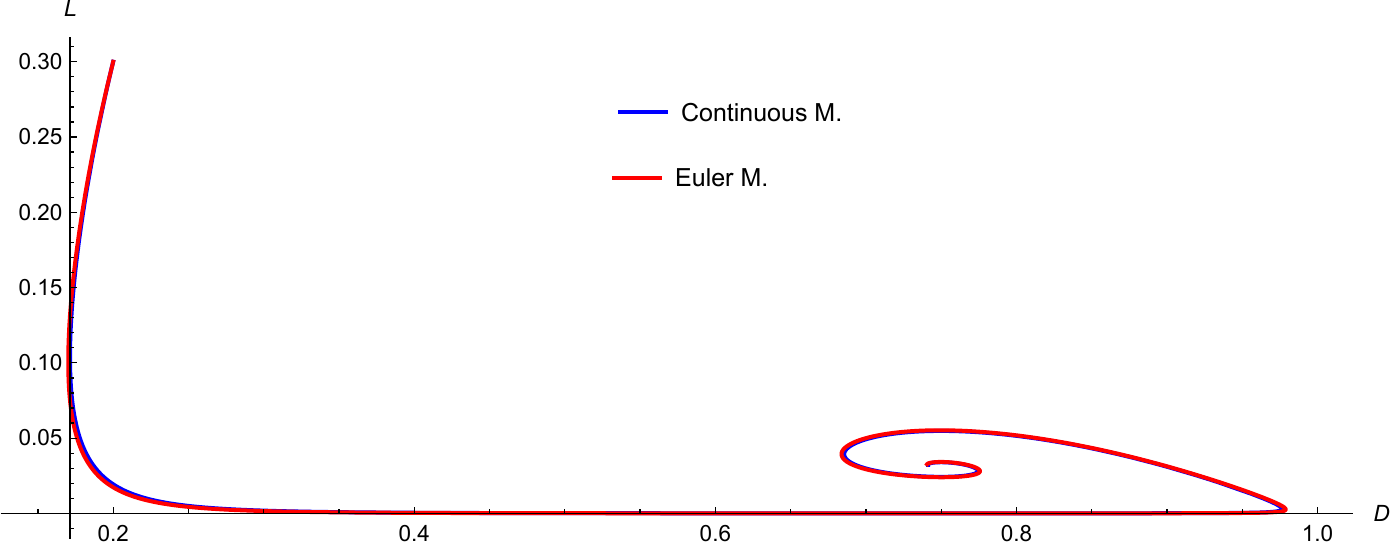}
\caption{Euler's method compared with the continuous model.\label{fig8}}
\end{center}
\end{figure}   
\end{center}


\section{Mickens' Numerical Scheme}
\label{sec4}

The nonstandard finite difference numerical scheme (NSFD), 
started by \cite{Mick94}, is based on Mickens' work  
\cite{Mick02, Mick05, Mick07, Mick13}. Is was created 
with the goal of solving some problems produced by Euler's method, 
namely numerical instabilities and the difficulty of showing 
the non-negativity of the solutions for physical models 
or dynamical inconsistency \cite{Mick02}. 

Mickens' numerical scheme has two main rules \cite{Mick02}. 
The first is: the derivative is approximated by 
$$
\dfrac{dx}{dt} 
\approx \dfrac{x_{k+1} - \varphi(h) x_k}{\psi(h)}, 
\quad h = \bigtriangleup t,
$$	
where $h$ is the step-size, and $\varphi(h)$ and $\psi(h)$ satisfy
$$ 
\varphi(h) = 1 + O(h^2) 
\quad \text{ and } \quad 
\psi(h) = h + O(h^2).
$$
	
The numerator function $\varphi(h)$ and the denominator function $\psi(h)$ may take
different forms. Generally,  $\varphi(h)  = 1$, but it can be different, 
for instance, $\varphi(h) = \cos(\lambda h)$; while $\psi(h)$ can be, for example, 
$\psi(h) = \dfrac{1 - e^{-\lambda h}}{\lambda}$, where $\lambda$ 
is a parameter that appears in the model. 

The second main rule of Mickens' numerical scheme is: 
the linear and nonlinear terms may require a non-local representation. 
For examples of such non-local representations see, e.g., 
\cite{Mick02,Mick05,Mick07}.


\subsection{Model discretization}	

Applying the Mickens rules,
\begin{align*}
\begin{cases}
\dot{D}(t) 
& \approx \dfrac{D((i+1)h) - D(ih)}{\phi(h)} 
= \dfrac{D_{i+1} - D_i}{\phi(h)}, \\ 
\dot{L}(t) 
& \approx \dfrac{L((i+1)h) - L(ih)}{\phi(h)} 
= \dfrac{L_{i+1} - L_i}{\phi(h)}.
\end{cases}	
\end{align*}
By \cite{Mick13}, if the populations have to satisfy a conservation law, 
then the denominator function is obtained by the conservation law 
of the continuous model \eqref{int}. We have seen that
\begin{equation*}
\dot{W} = \left(\dfrac{\alpha + 4 \beta }{4\beta}\right)M - \beta W,
\end{equation*}
so the denominator function is chosen to be
$\phi(h) = \dfrac{1 - e^{-\beta h}}{\beta}$. 
For simplification in writing, in this work
we denote $\phi := \phi(h)$.

The discrete model obtained from Mickens method is then given by
\begin{align}
&\begin{cases}
\dfrac{D_{i+1} - D_i}{\phi} 
&= \alpha D_i \left(1  - \dfrac{ D_{i+1}}{C} \right) - p D_{i+1}L_i,\\
\dfrac{L_{i+1} - L_i}{\phi} & = pD_{i+1} L_i - \beta L_{i+1},
\end{cases} \nonumber \\
& \Leftrightarrow
\begin{cases}
\label{mick}
D_{i+1} &= \dfrac{(\alpha \phi  + 1)D_i}{1 
+ p \phi L_i + \dfrac{\alpha \phi D_i}{C}},\\
L_{i+1} &= \dfrac{(p \phi D_{i+1} +1) L_i}{1 + \beta \phi},
\end{cases} 
\end{align}
for $i\geq 0$.
The explicit discrete model is
\begin{equation}
\begin{cases}
\label{micke}
D_{i+1} &= \dfrac{(\alpha \phi  + 1)D_i}{1 
+ p \phi L_i + \dfrac{\alpha \phi D_i}{C}},\\
& \\
L_{i+1} &= \left(\dfrac{p \phi (\alpha \phi  + 1)D_i }{1 
+ p \phi L_i + \dfrac{\alpha \phi D_i}{C}} +1\right)\dfrac{L_i}{1 + \beta \phi},
\end{cases} 
\end{equation}
for $i\geq 0$.


\subsection{Non-negativity of solutions and conservation law}

If the initial conditions are non-negative, by $(P_1)$, 
and the direct observation of \eqref{mick}, we have the 
non-negativity of the solutions. In the next result, 
for practical reasons, we consider $C = 1$. 

\begin{theorem}
Under hypotheses $(P_1)$ and $(P_2)$, 
if $C=1$ and $\xi=1+ \alpha \phi$, then
\begin{equation*}
\Omega = \left\{(D,L) \in (\mathbb{R}_{0}^{+})^2_+ 
: 0 \leq D_n \leq 1 \quad \text{and} \quad
W_n \leq 1 \leq \dfrac{4\alpha^2 +  \xi \beta^2}{4\alpha\beta} \right\},
\end{equation*}
where $W_{n}=D_{n}+L_{n}$.
\end{theorem}

\begin{proof}
\smartqed	
By the first equation of \eqref{mick}, we have 
\begin{equation*}
D_{n+1} = \dfrac{(1 + \alpha \phi)D_n}{1 + p \phi L_n + \alpha \phi D_n} 
< \frac{(1 + \alpha \phi)D_n}{1+ \alpha \phi D_n} 
\end{equation*}
for some  $n > 0$ and
\begin{equation*}
\underset{n \to + \infty}{\lim} 
\frac{(1 + \alpha \phi)D_n}{1+ \alpha \phi D_n} =1.
\end{equation*}
Thus, $0 \leq D_{n} \leq 1$. Adding the equations of  \eqref{mick} we have
\begin{align*}
\dfrac{W_{n+1} - W_n}{\phi} 
&= \alpha D_n - \alpha D_n D_{n+1} - \beta L_{n+1} 
- \beta D_{n+1} + \beta D_{n+1}\\
&= \alpha D_n + D_{n+1}( - \alpha D_n + \beta) - \beta W_{n+1}\\
& \leq \alpha D_n + \beta D_{n+1} \left( 1 
- \frac{\alpha}{ \xi \beta} D_{n+1} \right) - \beta W_{n+1}.
\end{align*}
Let $f(D) = \beta D \left(1 - \frac{D}{ K} \right)$,  where 
$K=\frac{\xi \beta}{\alpha}$. Then, $\overline{D} = \frac{K}{2}$ 
is the critical point and maximizer, so that $f(D)$ attains 
its maximum at $f(\overline{D}) = \frac{K \beta}{4}$.
Therefore,
\begin{align*}
\dfrac{W_{n+1} - W_n}{\phi} 
\leq \alpha + \dfrac{K \beta^2}{4} - \beta W_{n+1},\\
W_{n+1} \leq \dfrac{A}{1 +\beta \phi}  + \dfrac{W_n}{1 + \beta \phi}, 
\end{align*}
where $A=\alpha +\frac{\beta K}{4}=\frac{4 \alpha^{2} + \beta^{2} \xi}{4 \alpha}$. 
We conclude that
\begin{align*}
W_{n+1} 
&\leq \left(\dfrac{1}{1 + \beta \phi} \right)^{n+1}W_0 
+ \dfrac{ A \phi}{1 +\beta \phi} \left(\dfrac{1 
- \left(\frac{1}{1+\beta \phi}\right)^n}{1 - \frac{1}{1 + \beta \phi}}\right)\\
& \leq \left(\dfrac{1}{1 + \beta \phi} \right)^{n+1}W_0 
+ \dfrac{A}{\beta} \left(1 - \left[\dfrac{1}{1 + \beta \phi}\right]^n\right)
\end{align*}
and $\underset{n \to + \infty}{\lim}W_{n+1} 
= \dfrac{4\alpha^2 +  \xi \beta^2}{4\alpha \beta}$.
\qed
\end{proof}

Considering hypothesis $(P_{1})$, $\phi(h) = h + O(h^{2})$
and  $\xi =  1+ \alpha \phi$, 
and it is reasonable to consider $1< \xi < 2$. 


\subsection{Stability analysis}	

The stationary points of the discrete system \eqref{mick} are:
$$
\begin{cases}
F(D^*) = D^*\\ 
F(L^*) = L^*\\
\end{cases} 
\Leftrightarrow 
\begin{cases}
D^* = 0,\\ 
L^* = 0,\\
\end{cases} 
\vee 
\begin{cases}
D^* = C, \\ 
L^* = 0,\\
\end{cases} 
\vee 
\begin{cases}
D^* = \dfrac{\beta}{p}, \\ 
L^* = \dfrac{\alpha}{p}\left(1 - \dfrac{\beta}{pC}\right). 
\end{cases}
$$ 
The stationary points are equal to the ones 
of the continuous model \eqref{int}:
$$
e_1 =(0,0), 
\quad e_2 = (C,0), 
\quad e_3 = \left( \dfrac{\beta}{p}, \dfrac{\alpha}{p} 
\left(1 -\dfrac{\beta}{pC}\right)\right).
$$
Once more, to study the local stability, we linearize the system.
The Jacobian matrix of \eqref{micke} is given by
$$ 
J(D,L) 
= 
\left[ 
\begin{array}{ccc}
\dfrac{(1 + \alpha \phi)(1 + p\phi L )}{\left(1 
+ p \phi L + \dfrac{\alpha \phi D}{C}\right)^2}  
& \hspace*{0.1 cm}&- \dfrac{(1 + \alpha \phi) p \phi D}{\left(1 
+ p \phi L + \dfrac{\alpha \phi D}{C} \right)^2}\\ 
&&\\
\dfrac{p \phi L(1+\alpha \phi)(1+p \phi L)}{(1 + \beta \phi)\left(1 
+ p \phi L + \dfrac{\alpha \phi D}{C}\right)^2} 
& & \dfrac{1}{1 + \beta \phi}\left( 1
+ \dfrac{p \phi (1+\alpha \phi)
D \left(1+ \dfrac{\alpha \phi D}{C}\right)}{\left(1 
+ p \phi L + \dfrac{\alpha \phi D}{C}\right)^2}\right) \\
\end{array}  
\right].
$$

For $e_1=(0,0)$, the Jacobian matrix is
$$ 
J(e_1) 
= 
\left[ 
{\begin{array}{cc}
1 + \alpha \phi  & 0 \\ \\
0 & \dfrac{1}{1 + \beta \phi}  \\
\end{array}} 
\right]
$$
and
\begin{equation*}
P(\lambda) = \det (J(e_1) - \lambda Id) 
= (1 + \alpha \phi - \lambda) 
\left(\dfrac{1}{1 + \beta \phi} - \lambda\right). 
\end{equation*}
In this way, the eigenvalues are $ \lambda_1 = 1 + \alpha \phi$ 
and $\lambda_2 = \dfrac{1}{1 + \beta \phi}$. By $(P_1)$, we have $\lambda_1 > 1$ 
and $\lambda_2 < 1$, that is, $e_1$ is a \emph{saddle point}.

For $e_2 = (C,0)$, the Jacobian matrix is
$$ 
J(e_2) 
= 
\left[ 
{\begin{array}{ccc}
\dfrac{1}{1 + \alpha \phi}  
&&  -\dfrac{Cp \phi }{1 + \alpha \phi}\\
&&\\
0 && \dfrac{1 + Cp \phi }{1 + \beta \phi}\\
\end{array}} 
\right]
$$
and the characteristic polynomial is
\begin{equation*}
P(\lambda)  = \det (J(e_2) - \lambda Id) 
= \left(\dfrac{1}{1 + \alpha \phi}  - \lambda \right) 
\left(\dfrac{1 + Cp \phi }{1 + \beta \phi} - \lambda \right). 
\end{equation*}
The eigenvalues are $ \lambda_1 =\dfrac{1}{1 + \alpha \phi}$ 
and $\lambda_2 = \dfrac{1 + Cp \phi}{1 + \beta \phi}$. Once more, 
by $(P_1)$, $\lambda_1 < 1$ and $\lambda_2 > 1$. We conclude that 
$e_2$ is a \emph{saddle point}.

For $e_3 =\left( \dfrac{\beta}{p}, 
\dfrac{\alpha}{p} \left(1 - \dfrac{\beta}{pC} \right) \right)$, we have
$$ 
J(e_3) 
= \left[ 
\begin{array}{cc}
\dfrac{1 + \alpha \phi \left(1 - \frac{\beta}{pC}\right)}{1 + \alpha \phi}  
& -\dfrac{\beta \phi}{1 + \alpha \phi} \\
\dfrac{\alpha \phi \left(1-\frac{\beta}{pC}\right)\left(1 + \alpha 
\phi \left(1 - \dfrac{\beta}{pC}\right) \right)}{(1+ \alpha \phi) (1+ \beta \phi)}
&\dfrac{1}{1+\beta \phi}\left( 1 + \dfrac{\beta 
\phi (1+\frac{\alpha \phi \beta}{pC})}{1+\alpha \phi}\right)
\end{array}
\right]
$$
and, after some computations, it can be seen that
\begin{align*}
P(\lambda)  
&= \det (J(e_3) - \lambda Id)\\
&= \lambda^{2}-\left( 1+ \dfrac{(1+\beta \phi) 
+ \alpha \phi \left(1 - \frac{\beta}{pC}\right)}{(1
+ \beta \phi)(1+ \alpha \phi )}\right) \lambda 
+ \frac{1+\alpha \phi \left(1-\frac{\beta}{pC}\right)}{1+\alpha \phi}. 
\end{align*}
By the Schur--Cohn criterion for quadratic polynomials,  
if $P(1) > 0$, $ P(-1) > 0$ and $ \mid P(0) \mid < 1$, 
then all of its roots are in the unit circle \cite{Elaydi,Elaydi01}. 
By $(P_{1})$,
\begin{itemize}
\item $P(1) = \dfrac{\alpha \phi}{1 + \alpha \phi} 
\dfrac{\beta \phi}{1 + \beta \phi} \left(1 - \dfrac{\beta}{pC}\right) > 0$; 
\item $P(-1) = 2 + \left(\dfrac{1 + \alpha \phi 
\left(1 - \dfrac{\beta}{pC}\right)}{1 + \alpha \phi} \right) 
+ \dfrac{1+\beta \phi+\alpha \phi \left(1 
- \dfrac{\beta}{pC}\right) }{(1 + \alpha \phi)(1 + \beta \phi)} >0$.
\end{itemize} 	
Regarding 
$$
P(0)=\dfrac{1 + \alpha 
\phi \left(1 - \dfrac{\beta}{pC}\right)}{1 + \alpha \phi},
$$
by $(P_{1})$, $0<P(0)<1$. 

So, we conclude that $e_3$ is a \emph{sink} 
or an \emph{asymptotically stable point}.

We proved the boundedness of the solutions
to Mickens' numerical scheme, considering $C=1$, 
and the dynamical consistency with the continuous model 
without further restrictions.


\subsection{Graphical analysis}

Figure~\ref{fig4} presents solutions 
to the Mickens discrete model  \eqref{mick}
with different initial conditions.
\begin{center}
\begin{figure}[ht!]
\begin{center}
\includegraphics[width=6 cm]{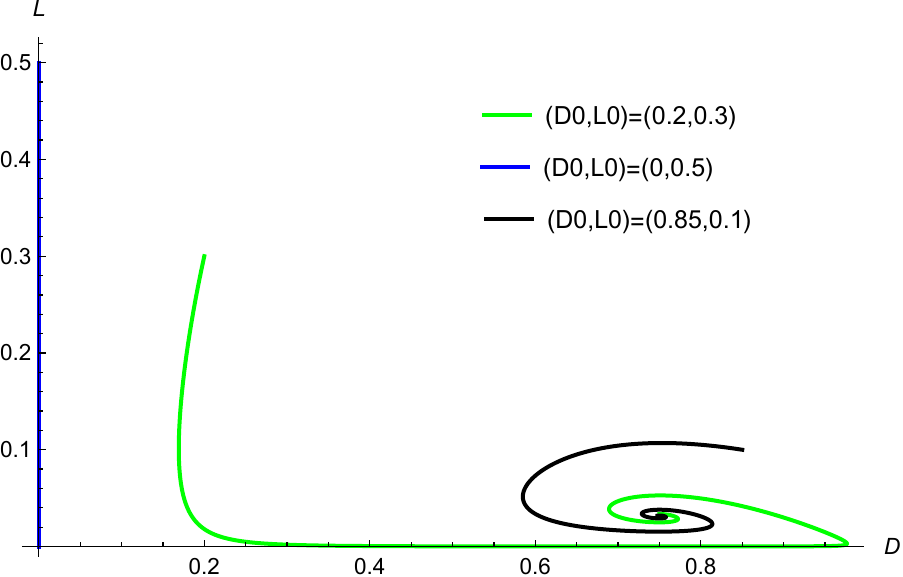}
\caption{Solutions to Mickens discrete model \eqref{mick} 
with different initial conditions. \label{fig4}}
\end{center}
\end{figure}   
\end{center}
Figure~\ref{fig4} is similar 
to Figures~\ref{fig2} and \ref{fig3}. 

Figure~\ref{fig7} compares the three previous models 
\eqref{int}, \eqref{EE} and \eqref{mick}.
As we can see, the differences between them
are only mild.
\begin{center}
\begin{figure}[ht!]
\begin{center}
\includegraphics[width=9 cm]{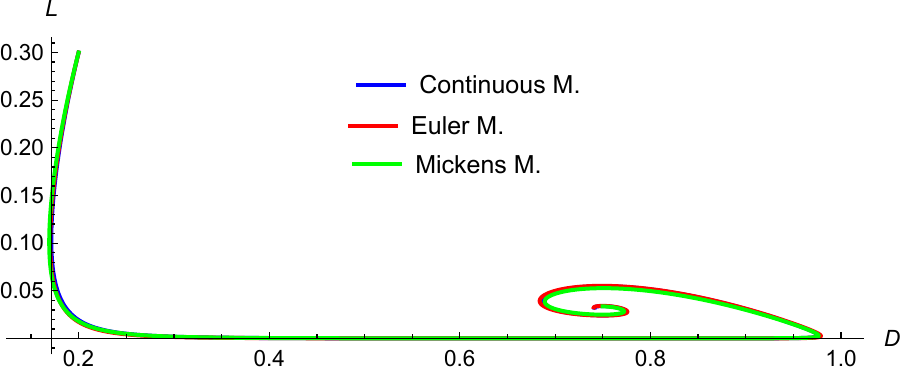}
\caption{Solutions of Mickens' model \eqref{mick} 
versus solutions of previous models 
\eqref{int} and \eqref{EE}. \label{fig7}}
\end{center}
\end{figure}   
\end{center}


\section{Fractional Calculus}
\label{sec5}

The calculus of non-integer order, known as Fractional Calculus (FC),  
is the branch of mathematics that studies the extension of the 
derivative and integral concepts to an arbitrary order, 
not necessarily of a fraction order \cite{R77}.

Considering $f(x) = \frac{1}{2} x^2$, the first and second derivative 
is $f'(x) = x$ and $f''(x) = 1$.  But what about the derivative of order 
$n = \frac{1}{2}$? This was the question that 
the father of fractional calculus, L'H\^{o}pital, considered, 
asking it, by letter, to Leibniz, in 1695. Since then, several mathematicians 
worked in such kind of calculus, namely  Gr\"unwald, Letnikov, Riemann, Liouville, 
Caputo, among many others \cite{Tese,MR3181071}.


\subsection{Preliminaries on FC}

In fractional calculus there are some functions, 
namely the Gamma, Beta and Mittag--Leffler functions, 
that play a crucial role.

\begin{definition}[Gamma function]
Let $z \in \mathbb{C}$. The gamma function is defined by
\begin{equation} 
\label{gama}
\Gamma (z) = \int_{0}^{\infty} e^{-t}t^{z-1} \, dt.
\end{equation} 
\end{definition}

The integral \eqref{gama} converges if  $Re(z) > 0$. 
The Gamma function has the following important property:  
$$
\Gamma (z+1) = z \Gamma (z).
$$

\begin{definition}[Beta function]
Let $z, w \in \mathbb{C}$. The beta function is defined by
\begin{equation}
\label{beta}
B(z,w) = \int_{0}^{1} \tau^{z-1}(1-\tau)^{w-1} \, d \tau
\end{equation} 
for $Re(w) >0$.
\end{definition}

Using the Laplace transform, we can rewrite  \eqref{beta} as
\begin{equation} 
\label{laplace}
B(z,w) = \dfrac{\Gamma (z) \Gamma (w)}{\Gamma (z+w)}.
\end{equation}
By \eqref{laplace} we conclude that $B(z,w) = B(w,z)$.

It is known that $e^z$, 
\begin{equation}
\label{eq:exp}
e^z = \sum_{k=0}^{\infty} \dfrac{z^k}{\Gamma (k+1)},
\end{equation}
has an important role in the integration 
of ordinary differential equations. Similar role has the 
Mittag--Leffler function for fractional differential equations.
The Mittag--Leffler function is a generalization of the exponential
function \eqref{eq:exp}.

\begin{definition}[Mittag--Leffler function]
The Mittag--Leffler function of two parameters is defined by
$$
E_{\alpha,\beta} = \sum_{k=0}^{\infty} \dfrac{z^k}{\Gamma (\alpha k + \beta)},
$$ 
where $\alpha, \beta \in \mathbb{C}$ and $Re(\alpha)>0$. 
\end{definition}

The Mittag--Leffler function is uniformly convergent in every 
compact subset of $\mathbb{C}$. If $\alpha = \beta = 1$, 
then $E_{1,1} (z) = e^z$. 

The fractional calculus has several formulations. 
The two most well-known are the Riemann--Liouville and Caputo
approaches. Here we use fractional derivatives in Caputo's sense.

To define the Riemann--Liouville fractional order integral,
let us begin by considering the integral of first order 
of a function $f$:
\begin{equation}
J^1 f(x) = \int_{0}^{x} f(x) \, dt.
\end{equation}
Analogously, let us consider the integral 
\begin{align*}
J^2 f(x) & = \int_{0}^{x} \int_{0}^{t_1} f(t) \, dt \, dt_1\\
& = \int_{0}^{x} \int_{t}^{x} f(t) \, dt_1 \, dt\\
& = \int_{0}^{x} f(t) \int_{t}^{x} 1 \, dt_1 \, dt\\
& = \int_{0}^{x} (x - t) f(t) \, dt. 
\end{align*}
By induction, it follows that
\begin{equation}
\label{int:int:ord:n}
J^{n+1} f(x) = \dfrac{1}{n!} 
\int_{0}^{x} (x-t)^n f(t) \, dt.
\end{equation}
The fractional order integral is obtained taking 
$n$ in \eqref{int:int:ord:n} to be a real number.

\begin{definition}[Riemann--Liouville fractional order integral]
Let $\alpha > 0$. The Riemann--Liouville fractional order 
integral of order $\alpha$ is given by
\begin{equation*}
J^\alpha f(x) = \dfrac{1}{\Gamma (\alpha)} 
\int_{0}^{x} (x - t)^{\alpha - 1} f(t) \, dt.
\end{equation*}
\end{definition}

The fractional derivative of Riemann--Liouville
is obtained with the help of the arbitrary order integral 
of Riemann--Liouville and the integer order derivative.

\begin{definition}[Riemann--Liouville fractional order derivative]
Let $\alpha > 0$, $m = \lceil \alpha \rceil$, and $v = m- \alpha$.
The Riemann--Liouville fractional order derivative of order $\alpha$
of function $f$ is defined by
\begin{align*}
D^\alpha f(x) & = D^m [J^v f(x)]\\
& = D^m \left(\dfrac{1}{\Gamma (\alpha)} 
\int_{0}^{x} (x - t)^{\alpha - 1} f(t) \, dt \right)\\ 
& = \dfrac{d^m}{dt^m} \left(\dfrac{1}{\Gamma (\alpha)} 
\int_{0}^{x} (x - t)^{\alpha - 1} f(t) \, dt \right)\\
& = \dfrac{1}{\Gamma (\alpha)} \dfrac{d^m}{dt^m} 
\int_{0}^{x} (x - t)^{\alpha - 1} f(t) \, dt.
\end{align*}
\end{definition}

When one uses the Laplace transform of the 
Riemann--Liouville fractional-order derivative,
initial values with integer-order derivatives 
are not obtained. This fact has no physical 
meaning \cite{bra}. For this reason, in many 
applications the Caputo fractional derivative is preferred.

The definition of Caputo's fractional derivative is similar 
to the  Riemann--Liouville definition, the difference being 
the order one takes the operations of Riemann--Liouville 
integration and integer-order differentiation. Indeed,
in Caputo's definition, first we compute the derivative 
of integer order and only then the fractional-order integral.

\begin{definition}[Caputo fractional order derivative]
Let $\alpha > 0$, $m = \lceil \alpha \rceil$,
and $v = m- \alpha$. The Caputo fractional order derivative 
of order $\alpha$ of function $f$ is defined by
\begin{align*}
{^c}D^\alpha f(x) 
&= J^v[D^m f(x)]  
= J^v \left[\dfrac{d^m}{dx^m} f(x) \right] 
= J^v f^{(m)} (x) \\ 
& = \dfrac{1}{\Gamma (v)} 
\int_{0}^{x} (x - t)^{v - 1} f^{(m)} (t) \, dt. 
\end{align*}
\end{definition}

Caputo's definition has two advantages with respect 
to the Riemann--Liouville definition: (i) in applications that 
involve fractional differential equations, the presence 
of initial values are physical meaningful; (ii) the derivative 
of a constant is zero, in contrast with the Riemann--Liouville.


\subsection{Model description}

Let $0 < \sigma < 1$, 
$(D(0), L(0)) \in (\mathbb{R}_{0}^{+})^2$, and
\begin{equation} 
\label{fraci}
\left\{\begin{array}{l}
{^c}D^\sigma D(t) 
= \alpha D \left(1 - \dfrac{D}{C}\right) - pDL,\\ 
{^c}D^\sigma L(t) = pDL - \beta L.
\end{array}\right.
\end{equation}
To write \eqref{fraci} in a compact way, let
\begin{equation*}
(\mathbb{R}_{0}^{+})^2 
= \{X \in \mathbb{R}^2 : X \geqslant 0\},
\end{equation*}
$X(t) = (\begin{matrix}
D(t) & L(t)
\end{matrix})^T$, and
\begin{equation}
\label{FF}
{^c} D^\sigma X(t) = F(X(t)), 
\quad X(0)=(D(0), L(0)) \in (\mathbb{R}_{0}^{+})^2,
\end{equation}
where
\begin{equation} 
\label{CV}
F(X) = \left(\begin{matrix}
\alpha D \left(1 - \dfrac{D}{C}\right) - pDL\\ 
pDL - \beta L
\end{matrix} \right) .
\end{equation}


\subsection{Existence and uniqueness of non-negative solutions}

Before the stability analysis, it is necessary to show the existence 
and uniqueness of non-negative solutions. 
The following lemma is important.

\begin{lemma}[See \cite{Delfim}]
\label{EU}
Let us assume that $F$ in \eqref{CV} satisfies the following conditions:
\begin{itemize}
\item $F(X)$ and $\frac{\partial F(X)}{\partial X}$ 
are continuous in $X \in \mathbb{R}^n$;
\item $\lVert F(X) \lVert \ \leqslant w + \lambda \lVert X \lVert$, 
$\forall X \in \mathbb{R}^n$, where $w$ and $\lambda$ are positive constants. 
\end{itemize}
Then a solution to \eqref{FF}--\eqref{CV} exists and is unique.
\end{lemma}

\begin{proposition}[Existence and uniqueness of non-negative solutions] 
\label{EUF}
There exists only one solution to the IVP  \eqref{FF} in
\begin{equation*}
(\mathbb{R}^+_0)^2 
= \left\{(D,L) \in \mathbb{R}^2 : (D(t), L(t)) \geqslant 0, \forall t >0\right\}.
\end{equation*}
\end{proposition}

\begin{proof}
\smartqed	
The existence and uniqueness of solution follows from Lemma~\ref{EU}.
The vector function \eqref{CV} is a polynomial so it is continuously
differentiable. Let
\begin{equation*}
Z = \left(
\begin{matrix}
-\dfrac{\alpha}{C} &  -p \\ 
0 & p
\end{matrix}\right), 
\quad 
B 
= \left(
\begin{matrix}
\alpha & 0\\ 
0 & -\beta
\end{matrix} \right).
\end{equation*}
Then,
\begin{equation}
\label{F}
F(X) = D ZX + BX.
\end{equation}
Using the $\sup$ norm,
\begin{align*}
\Vert F(X) \Vert & \leq \Vert ZX \Vert + \Vert BX \Vert \\
& = (\Vert Z \Vert + \Vert B \Vert) \, \Vert X \Vert \\ 
& < \epsilon + (\Vert Z \Vert + \Vert B \Vert) \, \Vert X \Vert	
\end{align*}
for some $\epsilon > 0$. So, by Lemma~\ref{EU}, system \eqref{FF} 
has a unique solution. The proof of the non-negative of the solutions  
follows the same idea of \cite{RA}. 
To prove that $(D^*_k(t),L^*_k(t)) \in (\mathbb{R}^+_0)^2$ for all $t \geqslant 0$,
let us consider the following auxiliary fractional differential system:
\begin{equation*} 
\begin{cases}
\begin{array}{l}
{^c}D^\sigma D(t) 
= \alpha D \left(1 - \dfrac{D}{C}\right) - pDL + \dfrac{1}{k},\\ 
{^c}D^\sigma L(t) = pDL - \beta L + \dfrac{1}{k},\\
\end{array}
\end{cases}	
\end{equation*}
with $k \in \mathbb{N}$. By contradiction, let us assume that exists 
a time instant where the condition fails. Let
\begin{equation*}
t_0 = \inf\left\{t>0 : (D^*_k(t),L^*_k(t)) \notin (\mathbb{R}^+_0)^2\right\}.
\end{equation*}
Then $(D^*_k(t_0),L^*_k(t_0)) \in (\mathbb{R}^+_0)^2$  and one 
of the quantities  $D^*_k(t_0)$ or $L^*_k(t_0)$ is zero. 
Let us suppose that  $D^*_k(t_0) = 0$. Then
\begin{equation*}
{^c}D^\sigma D^*_k(t_0) = 0 + \dfrac{1}{k} > 0.
\end{equation*}
By continuity of ${^c}D^\sigma D^*_k$, we conclude that 
${^c}D^\sigma D^*_k ([t_0,t_0 + \zeta]) \subseteq \mathbb{R}^+$ 
so $D^*_k$ is non-negative.  Analogously, we can do the same for  
${^c}D^\sigma L^*_k$, obtaining the intended contradiction. 
It follows by Lemma~1 of \cite{RA}, when $k \to \infty$, that 
$(D^*(t),L^*(t)) \in (\mathbb{R}^+_0)^2$ for all $t \geqslant 0$. 
\qed
\end{proof}


\subsection{The conservation law}

Considering Proposition~\ref{EUF}, 
there exists only one solution to the IVP  \eqref{FF}.
The proof of our next result 
uses some auxiliary results
found in \cite{RA,Delfim}. 

\begin{theorem}[Conservation Law]
\label{thm:consLaw}	
The solution to the IVP  \eqref{FF} is in 
\begin{equation*}
\Omega = \left\{(D,L) \in (\mathbb{R}_{0}^{+})^2: W(t) 
= D(t) + L(t) \leqslant W(0) + \frac{A}{\beta}\right\}, 		 	
\end{equation*}
where $A = \dfrac{\alpha + 4 \beta}{4} M$.
\end{theorem}

\begin{proof}
\smartqed	
Let
\begin{align*}
{^c}D^\sigma W(t) 
& \leqslant \alpha D(t) (1 - D(t)) - \beta L(t) - \beta D(t) + \beta D(t) \\
& \leqslant \dfrac{\alpha}{4} M + \beta M - \beta W(t)  
\leqslant \left(\dfrac{\alpha + 4 \beta}{4}\right)M - \beta W(t).
\end{align*}
It is known that 
\begin{equation*}
{^c} D^\sigma W(t) = J^{1-\alpha} \dot{W}(t) 
\quad \text{and} \quad 
\phi_\alpha (t) = \dfrac{t^{\alpha -1}}{\Gamma(\alpha)}, 
\quad \text{ for } t > 0.
\end{equation*}
Then
\begin{align*}
\mathcal{L}\{{^c}D^\sigma W(t)\} 
& = \mathcal{L}\{\phi_{1-\alpha}(t) \ast \dot{W}(t)\}  
= \mathcal{L}\{\phi_{1-\alpha}(t)\} \cdot \mathcal{L}\{\dot{W}(t)\}\\
& = s^{\alpha-1} (sW(s) - W(0))= s^\alpha W(s) - s^{\alpha-1} W(0),
\end{align*}
so that
\begin{align*}
{^c}D^\sigma W(t) + \beta W(t) \leq A  
&\Leftrightarrow \mathcal{L}\{{^c}D^\sigma W(t) 
+ \beta W(t)\} \leq \mathcal{L}\left\{A\right\} \\
& \Leftrightarrow s^\alpha W(s) - s^{\alpha-1}W(0) 
+ \beta W(s)  \leq \dfrac{A}{s} \\
& \Leftrightarrow W(s)  \leq \dfrac{A}{s(s^\alpha + \beta)} 
+ \dfrac{s^{\alpha-1}}{s^\alpha + \beta} W(0). 
\end{align*}
Therefore, 
\begin{align*}
W(t) &= \mathcal{L}^{-1} \{W(s)\} \\
& \leq \mathcal{L}^{-1} \left\{ \dfrac{A}{s(s^\alpha + \beta)} \right\} 
+ \mathcal{L}^{-1} \left\{\dfrac{s^{\alpha -1}}{s^\alpha + \beta}W(0)\right\} \\
&= A \mathcal{L}^{-1} \left\{\dfrac{1}{s(s^\alpha + \beta)} \right\} 
+ W(0) \mathcal{L}^{-1} \left\{\dfrac{s^{\alpha-1}}{s^\alpha + \beta}\right\} \\
&= \dfrac{A}{\beta} (1 - E_\alpha(-\beta t^\alpha)) + W(0) E_\alpha(-\beta t^\alpha).
\end{align*}
Since $0 \leqslant E_\alpha(-\beta t^\alpha) \leqslant 1$, 
we have $W(t) \leqslant W(0) + \frac{A}{\beta}$.
\qed
\end{proof}


\subsection{Stability analysis}
\label{sub:sec:SA}

The equilibrium points of the fractional system \eqref{fraci} are
$$
\left\{
\begin{array}{ll}
{^c}D^\alpha D(t) = 0\\ 
{^c}D^\alpha L(t) = 0\\
\end{array}\right. 
\Leftrightarrow 
\left\{
\begin{array}{ll}
D = 0,\\
L = 0,\\
\end{array}\right.
\vee 
\left\{
\begin{array}{ll}
D = C, \\ 
L = 0,\\
\end{array}\right. 
\vee 
\left\{
\begin{array}{ll}
D = \dfrac{\beta}{p}, \\ 
L = \dfrac{\alpha}{p} \left(1 - \dfrac{\beta}{pC}\right).
\end{array}\right.
$$ 
They are similar to the ones of the continuous model \eqref{banc}: 
$$
e_1 =(0,0), 
\quad e_2 = (C,0), 
\quad e_3 = \left(\dfrac{\beta}{p}, 
\dfrac{\alpha}{p}\left(1 - \dfrac{\beta}{pC}\right)\right),
$$
and the study of their local stability is equal to the continuous model.
Precisely:
\begin{itemize}
\item $e_1$  and $e_{2}$ are \emph{saddle points};
\item $e_3$ is a \emph{sink} or an \emph{asymptotically stable point}.
\end{itemize} 


\subsection{Graphical analysis}
\label{sub:sec:GA}

We now present some graphics obtained by the modified trapezoidal method, 
for $\sigma = 0.95$, step-size $h = 0.25$, and the time interval $[0,300]$.

In Figure~\ref{fig5} we observe that the orbit with initial conditions   
$(D_0,L_{0}) =(0.2, 0.3)$ tend to the equilibrium point faster 
than the other methods. The same happens 
with $(D_{0}, L_{0})=(0.85,0.1)$. 
\begin{center}
\begin{figure}[ht!]
\begin{center}
\includegraphics[width=7cm]{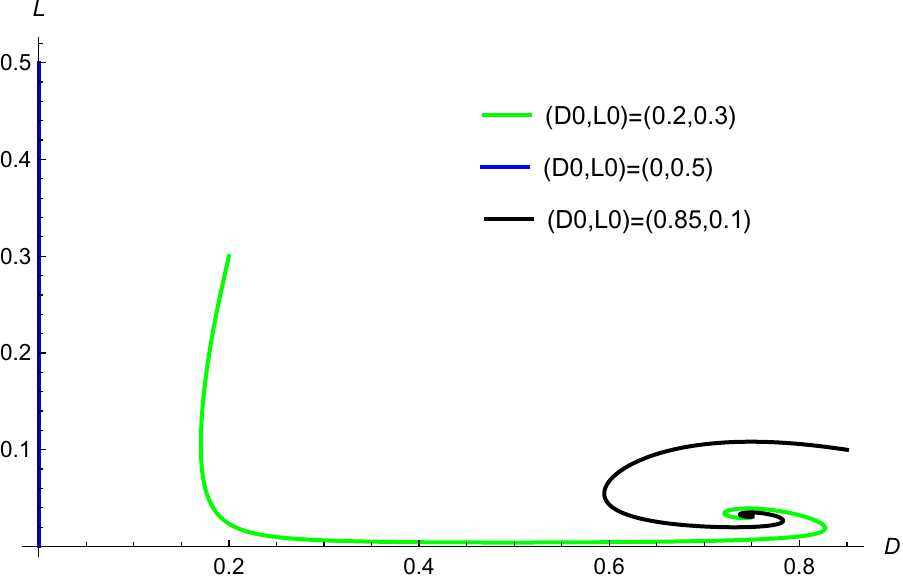}
\caption{Solutions to the fractional model \eqref{fraci}
with different initial conditions. \label{fig5}}
\end{center}
\end{figure}   
\end{center}

Figure~\ref{fig6} shows that the equilibrium point 
is attained faster and Figure~\ref{fig10}  
illustrates, graphically, that when $\sigma$ tends to $1$ 
the solution of \eqref{fraci} tends to the solution
of the continuous model \eqref{banc}, showing that the modified 
trapezoidal method is a good discretization method to our problem.
\begin{center}
\begin{figure}[t]
\begin{center}
\includegraphics[width=9cm]{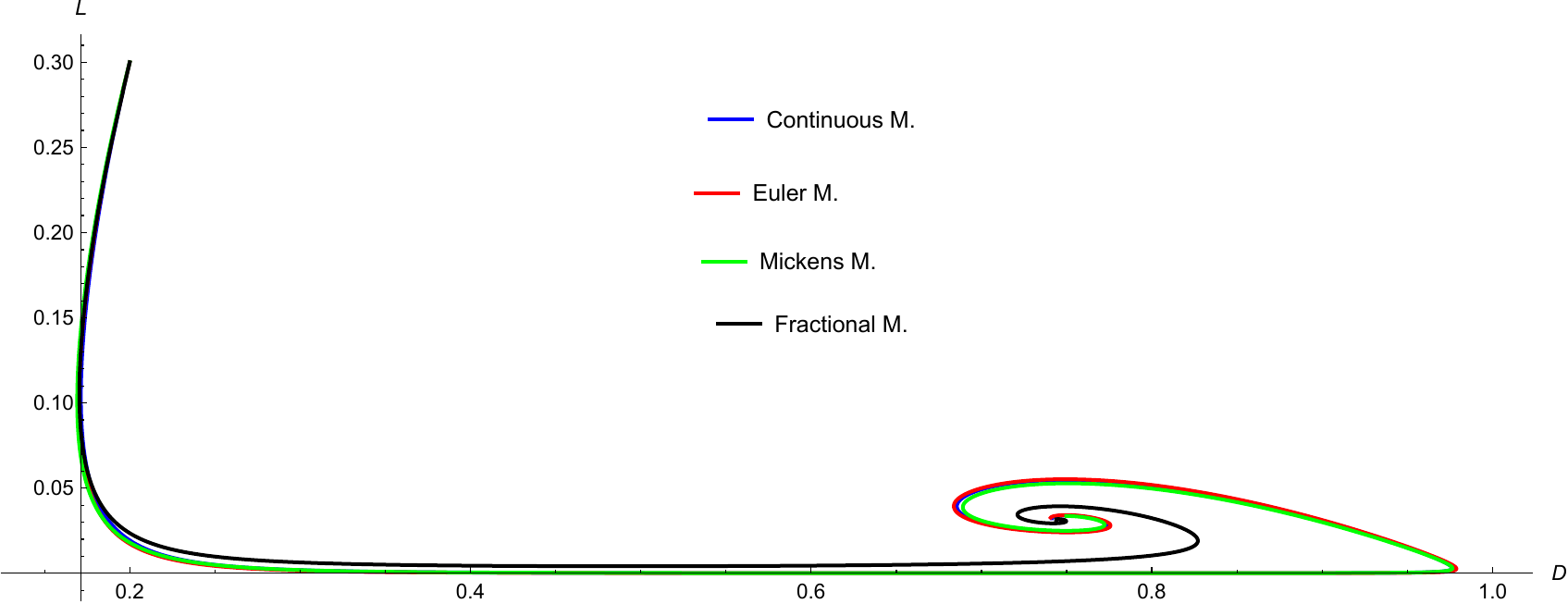}
\caption{Solution to fractional model compared 
with solutions to continuous and discrete Euler's 
and Micken's models. \label{fig6}}
\end{center}
\end{figure}   
\end{center}
\begin{center}
\begin{figure}[t]
\begin{center}
\includegraphics[width=9cm]{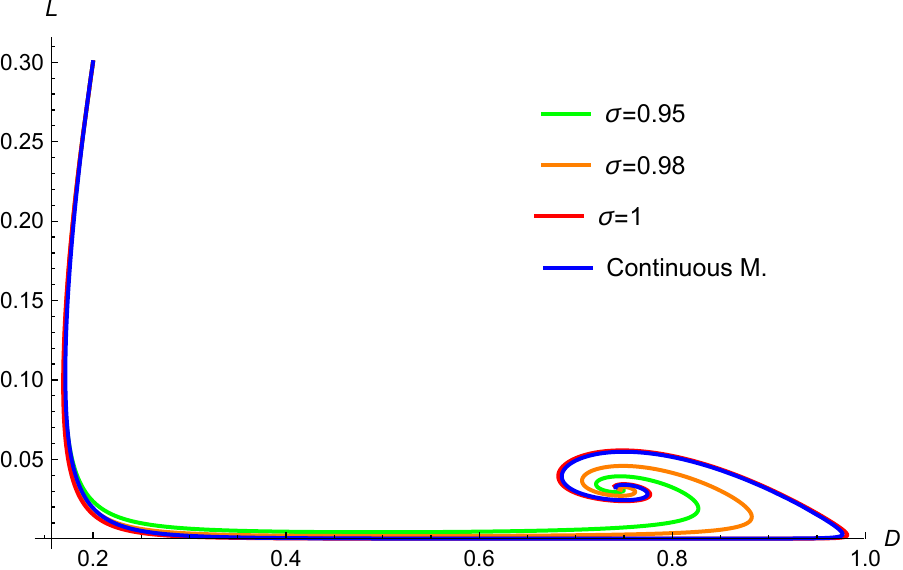}
\caption{Solution to fractional system \eqref{fraci}, 
for different values of $\sigma$, compared with the solution 
to the continuous model \eqref{banc}. \label{fig10}}
\end{center}
\end{figure}   
\end{center}


\section{Conclusion}
\label{sec6}

In this work we considered a predator-prey Lotka--Volterra 
type model and analyzed it through different methods. 
The use of the logistic function, to explain the growth 
of the prey population in the absence of predation, 
and the study of the model from different perspectives,  
allow us to know when extra conditions are needed 
to achieve dynamical consistency. 
We were able to apply Euler numerical scheme, under some extra conditions, 
namely $1-\beta h > 0$ and $ h \in \, ]0, \frac{1}{pC -\beta}[$. 
Regarding Mickens' nonstandard finite difference numerical scheme, 
the proofs were carried out using the carrying capacity $C=1$. 
Writing the model using fractional calculus,
and using the same initial conditions, our graphical analysis allowed us 
to conclude that with $\sigma=0.95$ the equilibrium point is attained faster. 


\section*{Acknowledgements}

The authors were partially supported by 
the Portuguese Foundation for Science and Technology (FCT):
Vaz through the Center of Mathematics and Applications 
of \emph{Universidade da Beira Interior} (CMA-UBI), 
project UIDB/00212/2020; Torres through
the Center for Research and Development in Mathematics 
and Applications (CIDMA), grants UIDB/04106/2020
and UIDP/04106/2020, and within the project ``Mathematical Modelling 
of Multi-scale Control Systems: Applications to Human Diseases'' (CoSysM3), 
reference 2022.03091.PTDC, financially supported by national funds (OE) 
through FCT/MCTES.



\end{document}